\documentclass{article}

\usepackage{graphicx}
\usepackage{amssymb,amsmath,amscd,amstext,amsfonts, amsthm}
\usepackage{xypic}
\input{xy}
\xyoption{all}

\newtheorem{thm}{Theorem}[section]

\newtheorem{prop}[thm]{Proposition}
\newtheorem{defn}[thm]{Definition}
\newtheorem{lem}[thm]{Lemma}

\newcommand {\Z} {\mathbb Z}
\newcommand {\Q} {\mathbb Q}
\newcommand {\ZG} {\mathbb {Z}[G]}

\newcommand {\rk} {{\rm rk}_{\Z}}

\begin{document}

\begin{center} {\bf \huge Duality in the homology of 5-manifolds}

{W.H. Mannan}\end{center}

\begin{abstract} We show that the homological properties of a 5-manifold $M$ with  fundamental group $G$ are encapsulated in a $G$--invariant stable form on the dual of the third syzygy of $\mathbb{Z}$.  In this notation one may express an even stronger version of  Poincar\'e duality for $M$.  However we find an obstruction to this duality.
\end{abstract}

\noindent
{Department of Mathematics, 
City University London,
Northampton Square, 
London, 
EC\textup{1V 0}HB \qquad
{\bf Tel}: 07951825081 \qquad
{\bf email}: {wajid@mannan.info}}



\section{Introduction}   \label{intro}

The purpose of this article is to  consider a 5-manifold $M$ with finite fundamental group $G$, and examine the interaction between the module theory over $\ZG$  with the homology theory of the  simply connected 5-manifold $\tilde M$.  Much of what we do holds more generally for $2n+1$-manifolds with $n-1$-connected universal cover.  However we focus on 5-manifolds for clarity and in order to avoid such additional assumptions, as well as to highlight that low-dimensional phenomena to not impinge on our results.

We associate to $M$ an algebraic complex $C(M)$; the cellular chain complex for $\tilde{M}$ which one may regard as an algebraic complex over $\ZG$.  This algebraic object (defined up to chain homotopy equivalence) is universal in the sense that all twisted homology/cohomology of $M$ may be extracted from it (by taking coefficients in the appropriate module over $\ZG$.

Using purely algebraic methods from our earlier paper \cite{Mann2}, in \S\ref{dualsec} we show that we may take $C(M)$ to satisfy a certain duality, so that its modules and those of its dual $C(M)^*$ may be identified with four of the five differentials the same in both cases.  We call this {\it dual form}.  Note this does not imply $C(M)\cong C(M)^*$.

A self-dual (up to signs on boundary maps) polyhedral cell description for $M$ itself could be obtained from Lefschetz' classical proof of Poincar\'e Duality cf. \cite[Theorem 2.1]{Wall} or via geometric arguments cf \cite[Theorem 6.5]{Klei}.  Our result differs in that we are working with the universal cover of $M$ and modules over $\ZG$.  As mentioned earlier this object more completely represents the (co)homology of $M$.

In \S\ref{Poincaresec} we observe that when $C(M)$ is in dual form, the equivalence induced by Poincar\'e duality may be taken to be plus/minus the identity on four of the six modules.  However some redundancy in the proof of this, together with evidence on the level of the derived category of $\ZG$ on the one hand, and the homology of $\tilde{M}$ on the other, suggest that we may extend this to all six modules. 

 In this case $C(M)$ would satisfy a stronger duality and we say $M$ satisfies {\it anti-self-duality}.  Here the modules in $C(M)$ and $C(M)^*$ are again identified, with four of the five differentials the same, but now the remaining differential differs by a sign in the two cases.  Finally we find an obstruction to anti-self-duality and show that not all 5-manifolds satisfy it, whilst providing a class of examples which do.

{\bf Notation}:  Fix a finite group $G$.  All modules are right modules unless otherwise stated and assumed to be over $\ZG$, as are linear maps between modules.  Bilinear forms on modules are assumed to be $G$ invariant.  We abbreviate finitely generated to fg.

By algebraic complex we mean a sequence of modules and maps (referred to as {\it differentials}) where the composition of successive terms is 0.  The maps will be represented by solid arrows.  Dashed arrows will denote identifications of the kernel/cokernel of the maps at either end with a module.  Dashed arrows should not be regarded as part of the complex (for homotopies between chain maps etc.).

By a \emph{simple homotopy equivalence} we mean interchanging an algebraic complex with one which differs from it in only two consecutive modules, where isomorphic free modules have been added to both modules, and the differential between them has been extended by the isomorphism.  The preceding/following differentials are composed with the natural inclusion/projection respectively.  If two algebraic complexes are related in this way, then we have a projection chain map and an inclusion chain map between them, which form a pair of mutually inverse chain homotopy equivalences.

\begin{defn}\label{stabdef}
For a finite  algebraic complex $\mathcal{A}$, the $n^{\rm th}$ \emph{stabilization} of $\mathcal{A}$, denoted $\mathcal{A}^n$, is the algebraic complex obtained by adding a free rank $n$ module $F$ to the leftmost module in $\mathcal{A}$ (with maps written going from left to right).  The differential from this module is extended to $0$ on $F$.
\end{defn}

We define the dual of a right $\ZG$  module $X$ to be the right $\ZG$ module $X^*$ obtained by taking the left $\ZG$ module ${\rm Hom}_\Z(X,\Z)$ and defining the right action of $g\in G$ to be left multiplication by $g^{-1}$.  

We extend this definition in the natural way to linear maps and algebraic complexes.  Note that dualization preserves exactness at a module in an algebraic complex of torsion free (over $\Z$) fg $\ZG$ modules.  We {do not} introduce signs when we dualize differentials.  The price for this is a sign factor in the statement of Poincar\'e duality (see (\ref{chaindef})).  Define $H^i(\mathcal{A})=H_i(\mathcal{A}^*)$ for an algebraic complex $\mathcal{A}$.

The integers $\Z$ are a module over $\ZG$ with trivial $G$ action.  A bilinear form  $X\times X \to \Z$ will be referred to interchangeably as a map $X \to X^*$.

Let ${\rm TOP}^5$ denote the category of closed connected oriented five dimensional 
topological manifolds with base point, equipped with an identification of their fundamental group with $G$. Let the morphisms in ${\rm TOP}^5$ be the continuous maps which preserve the base point and induce the identity on $G$.

 Given an object $M \in{\rm TOP}^5$ we may find a finite CW complex $M'$ which is homotopy equivalent to it 
\cite[Ann. 2]{Kirb}.  Let $C_*(\tilde{M'})$ be the (cellular) chain complex of $\tilde{M'}$:
\begin{eqnarray*}
C_*(\tilde{M'})\!:=C_5 \stackrel{\partial_5}{\longrightarrow} C_4 \stackrel{\partial_4}{\longrightarrow}C_3 \stackrel{\partial_3}{\longrightarrow}
C_2 \stackrel{\partial_2}{\longrightarrow} C_1 \stackrel{\partial_1}{\longrightarrow}C_0.
\end{eqnarray*}
This is an algebraic complex of fg free $\ZG$ modules and $\ZG$--linear maps.  It is exact at $C_1$ and coker$(\partial_1)\cong \Z$ (as $M$ connected).  We fix this isomorphism by fixing the homology class represented by a point to correspond to $1 \in \Z$.

Poincar\'e duality states that $C_*(\tilde{M'})$ is chain homotopy equivalent to its dual $C_*(\tilde{M'})^*$.  Therefore 
it is exact at $C_4$ and ker$(\partial_5)\cong\Z$.  We fix this isomorphism by demanding the generator of ker$(\partial_5)$ preferred by the orientation, maps to $1\in \Z$.

The Euler characteristic (alternating sum of torsion free ranks over $\Z$)
of $C_*(\tilde{M'})$ is minus that of $C_*(\tilde{M'})^*$ (as dualization preserves the torsion free $\Z$ rank of an fg $\ZG$ module).  However if they are chain homotopy equivalent they must also have equal Euler characteristic,  which must then be $0$.

\begin{defn}\label{algcat}
 Let ${\rm ALG}^5$ denote the category of algebraic 5--complexes of fg free $\ZG$ modules{\rm:} 
\[
\Z\dashrightarrow F_5 \stackrel{\partial_5}{\longrightarrow} F_4 \stackrel{\partial_4}{\longrightarrow}F_3 \stackrel{\partial_3}{\longrightarrow}
F_2 \stackrel{\partial_2}{\longrightarrow} F_1 \stackrel{\partial_1}{\longrightarrow}F_0 \dashrightarrow \Z,
\]

\noindent satisfying exactness at $F_4$ and $F_1$, and with Euler characteristic $0$.  Define the morphisms in this category to be homotopy classes of chain maps.
\end{defn}

We may define a functor $C\colon {\rm TOP}^5 \to {\rm ALG}^5$ by choosing a homotopy equivalence, $h_{M}\colon M \to M'$ with $M'$ a finite CW complex, for each $M \in{\rm TOP}^5$.  In each case fix also $h_M'$ a homotopy inverse of $h_M$.  Then define $C(M)$ to be $C_*(\tilde{M'})$.   

Given a morphism in ${\rm TOP}^5$, $f\colon M_1 \to M_2$, 
we may select a cellular map $f'\colon M_1' \to M_2'$, which is homotopic to $(h_{M_2} \circ f \circ h_{M_1}')$.  Define $C(f)$ to be the equivalence class of the chain map $f'_*\colon C_*(\tilde{M_1')} \to C_*(\tilde{M_2')}$.

The isomorphism class of $C(M)$ in ${\rm ALG}^5$ is an invariant of $M$ as by construction different choices of $M'$ must be
homotopy equivalent to each other.

\section{Dual form}\label{dualsec}

Consider an algebraic 2--complex $\mathcal{B}$ of free fg $\ZG$ modules, where the cokernel of the last map is identified with $\Z$, we have exactness at the middle term and the  dual of the kernel of the first map is denoted $J$:
\[\mathcal{B}:=\qquad J^* \stackrel{\iota}\dashrightarrow F_2\stackrel{d_2}{\to} F_1 \stackrel{d_1}{\to} F_0 \dashrightarrow \Z.\] 
Thus $J^* \in \Omega_3(\Z)$ the third syzygy of $\Z$ \cite[p6]{John1}.  Equip $\mathcal{B}$ with a $G$ invariant bilinear form $\beta$ on $J$.   We can associate to the pair $(\mathcal{B},\beta)$ an algebraic 5--complex:
\begin{eqnarray*}
i(\mathcal{B},\beta):=\qquad \Z\dashrightarrow F_0^* \stackrel{d_1^*}\to F_1^* \stackrel{d_2^*}{\to} F_2^* \,\,\,\stackrel{d_3} \longrightarrow\,\,\,\, F_2 \stackrel{d_2}\to F_1 \stackrel{d_1}\to F_0 \dashrightarrow\Z. \quad\, \\
\iota^*\searrow \quad\,\,\,\,\quad \nearrow \iota \qquad \qquad\qquad\,\,\quad\quad\,\,\,\\
J \stackrel{\beta}\to J^* \qquad \qquad \qquad\,\,\,\,\,\quad\quad\,\,\quad
\end{eqnarray*}
 Let ${\rm DUAL}^2$ denote the category whose objects are pairs $(\mathcal{B},\beta)$ as above.

We define the morphisms of ${\rm DUAL}^2$ to be homotopy equivalence classes of chain maps
between the associated algebraic 5--complexes.  If an object of ${\rm ALG}^5$ is of the form $i(\mathcal{B},\beta)$, for some $(\mathcal{B},\beta)\in {\rm DUAL}^2$ then we say it is in \emph{dual form}.

\begin{thm}\label{algtodual}{The functor $i$ is a natural equivalence of categories}
\end{thm}

By construction $i\colon{\rm DUAL}^2 \to {\rm ALG}^5$ is a full faithful functor.  It remains to show that every object in ${\rm ALG}^5$ is isomorphic to an object in the image of $i$.  Our proof will follow stages analogous to the proof of \cite[Theorem 1.1]{Mann2}.

  Consider an arbitrary element of ${\rm ALG}^5$:
\begin{eqnarray}\label{algtodual1}
\Z \dashrightarrow C_5 \stackrel{\partial_5}{\longrightarrow} C_4 \stackrel{\partial_4}{\longrightarrow}C_3 \stackrel{\partial_3}{\longrightarrow}
C_2 \stackrel{\partial_2}{\longrightarrow} C_1 \stackrel{\partial_1}{\longrightarrow}C_0 \dashrightarrow \Z. 
\end{eqnarray}

 We perform three pairs of simple homotopy equivalences.  Firstly, the complex (\ref{algtodual1}) is chain homotopy equivalent to:
\begin{eqnarray}\label{algtodual2}
C_5 \oplus C_0^*\stackrel{\delta_5}{\longrightarrow} C_4 \oplus C_0^*\stackrel{\partial_4 \oplus 0}{\longrightarrow}C_3 
\stackrel{\partial_3}{\longrightarrow}
C_2 \stackrel{\partial_2}{\longrightarrow} C_1 \oplus C_5^* \stackrel{\delta_1}{\longrightarrow}C_0 \oplus C_5^*,
\end{eqnarray}
where
\[
\delta_1 = \left( \begin{array}{cc} \partial_1&0\\ 0&1 \end{array} \right), \qquad\qquad
\delta_5 = \left( \begin{array}{cc} \partial_5&0\\ 0&1 \end{array} \right).
\]

 Let $R_0=C_0$, $R_5=C_5$, and $R_1 = C_1 \oplus C_5^*$, $R_4 = C_4 \oplus C_0^*$.  Rewrite (\ref{algtodual2}):
\begin{eqnarray}\label{algtodual3}
R_5 \oplus R_0^*\stackrel{\delta_5}{\longrightarrow} R_4 \stackrel{\partial_4 \oplus 0}{\longrightarrow}C_3 
\stackrel{\partial_3}{\longrightarrow}
C_2 \stackrel{\partial_2}{\longrightarrow} R_1 \stackrel{\delta_1}{\longrightarrow}R_0 \oplus R_5^*.
\end{eqnarray}

Again we perform a pair of simple homotopy equivalences.  The complex (\ref{algtodual3}) is chain homotopy equivalent to:
\begin{eqnarray}\label{algtodual4}
R_5 \oplus R_0^*\stackrel{\delta_5}{\longrightarrow} R_4 \oplus R_1^*\stackrel{\delta_4}{\longrightarrow}C_3 \oplus R_1^* 
\stackrel{\partial_3 \oplus 0}{\longrightarrow}
C_2 \oplus R_4^* \stackrel{\delta_2}{\longrightarrow} R_1 \oplus R_4^*\stackrel{\delta_1 \oplus 0}{\longrightarrow}R_0 \oplus R_5^*, 
\end{eqnarray}
where 
\[
\delta_2 = \left( \begin{array}{cc} \partial_2&0\\ 0&1 \end{array} \right), \qquad\qquad
\delta_4 = \left( \begin{array}{cc} \partial_4 \oplus 0 &0\\ 0&1 \end{array} \right).
\]

 Let $R_2 = C_2 \oplus R_4^*$, $R_3 = C_3 \oplus R_1^*$.  Then (\ref{algtodual4}) can be written:
\begin{eqnarray}\label{algtodual5}
R_5 \oplus R_0^*\stackrel{\delta_5}{\longrightarrow} R_4 \oplus R_1^*\stackrel{\delta_4}{\longrightarrow}R_3  
\stackrel{\partial_3 \oplus 0}{\longrightarrow}
R_2 \stackrel{\delta_2}{\longrightarrow} R_1 \oplus R_4^*\stackrel{\delta_1 \oplus 0}{\longrightarrow}R_0 \oplus R_5^*. 
\end{eqnarray}
As all the modules in this complex are free and the Euler characteristic is 0, we can assume the existence of some
isomorphism $\theta\colon R_2^* \to R_3^*$.

We perform a final simple homotopy equivalence to get:
\begin{eqnarray}\label{algtodual6}
R_5 \oplus R_0^*\stackrel{\delta_5}{\longrightarrow} R_4 \oplus R_1^*\stackrel{\delta_4}{\longrightarrow}R_3 \oplus R_2^* 
\stackrel{\delta_3}{\longrightarrow}
R_2 \oplus R_3^*\stackrel{\delta_2 \oplus 0}{\longrightarrow} R_1 \oplus R_4^*\stackrel{\delta_1 \oplus 0}{\longrightarrow}R_0 
\oplus R_5^*,
\end{eqnarray}
where: 
\[
\delta_3 = \left( \begin{array}{cc} \partial_3 \oplus 0&0\\ 0&\theta \end{array} \right). \qquad\qquad
\]

Our initial algebraic complex (\ref{algtodual1})
 is therefore chain homotopy equivalent to (\ref{algtodual6}).  We will show that (\ref{algtodual6}) is chain isomorphic to an algebraic 
5--complex in dual form.

\begin{lem}{There exist a pair of inverse chain isomorphisms $h,k$:}
\begin{eqnarray}\label{endcomplex}
R_2 \oplus R_3^*\stackrel{\delta_2 \oplus 0}{\longrightarrow} R_1 \oplus R_4^*\stackrel{\delta_1 \oplus 0}{\longrightarrow}R_0 
\oplus R_5^* \stackrel{\epsilon \oplus 0}{\dashrightarrow} \Z \to 0\\
h_2\downarrow\uparrow{k_2} \qquad {h_1}\downarrow\uparrow{k_1}\qquad\quad h_0\downarrow\uparrow{k_0} \qquad \vert\vert 1\quad\,\,\, \nonumber\\
\label{otherendcomplexdual}
R_3^* \oplus R_2\stackrel{\delta_4^* \oplus 0}{\longrightarrow} R_4^* \oplus R_1\stackrel{\delta_5^* \oplus 0}{\longrightarrow}R_5^* 
\oplus R_0 \stackrel{\epsilon' \oplus 0}{\dashrightarrow} \Z \to 0
\end{eqnarray}
\end{lem}

\begin{proof}
This is just a case of \cite[Lemma 2.1]{Mann2}.
 \end{proof}

 Let $d_3 = \delta_3 k_2^*$ and write (\ref{endcomplex}) as: \[  S_2 \stackrel{d_2}\to S_1 \stackrel{d_1}\to S_0.\]

\begin{lem}{ The complex (\ref{algtodual6}) is chain isomorphic to: 
\[
S_0^* \stackrel{d_1^*}{\longrightarrow} S_1^* \stackrel{d_2^*}{\longrightarrow} S_2^* \stackrel{d_3}{\longrightarrow} 
S_2 \stackrel{d_2}{\longrightarrow} S_1 \stackrel{d_1}{\longrightarrow} S_0. 
\]
}
\end{lem}

 \begin{proof}  We have the following chain isomorphism:
\begin{eqnarray*}
R_5 \oplus R_0^*\stackrel{\delta_5}{\longrightarrow} R_4 \oplus R_1^*\stackrel{\delta_4}{\longrightarrow}R_3 \oplus R_2^* 
\stackrel{\delta_3}{\longrightarrow}
R_2 \oplus R_3^*\stackrel{\delta_2 \oplus 0}{\longrightarrow} R_1 \oplus R_4^*\stackrel{\delta_1 \oplus 0}{\longrightarrow}R_0 
\oplus R_5^* \quad\\
\downarrow h_0^* \qquad \qquad \downarrow h_1^* \,\,\, \quad \qquad \downarrow h_2^* 
\qquad \qquad\downarrow 1 \qquad \qquad\downarrow 1 \qquad \qquad \downarrow 1 \qquad 
\\
\quad S_0^* \quad\stackrel{d_1^*}{\longrightarrow} \quad \quad S_1^* \quad \stackrel{d_2^*}{\longrightarrow} \quad S_2^*  \quad
\stackrel{d_3}{\longrightarrow} \quad\quad
S_2 \quad \stackrel{d_2}{\longrightarrow} \quad S_1 \quad \stackrel{d_1}{\longrightarrow} \quad \quad S_0  \quad\quad
\end{eqnarray*}
We may verify that the central square commutes: $d_3 h_2^* = \delta_3 k_2^* h_2^* = \delta_3$.
 \end{proof}

Clearly $d_3$ is induced by a $G$ invariant bilinear form on the dual of ker$(d_2)$.  This then completes the proof of Theorem \ref{algtodual}.

 Thus when using the functor $C$ to provide an invariant (up to isomorphism) of an element of ${\rm TOP}^5$,  we may work in the category ${\rm DUAL}^2$.  From now on we will suppress the functor $i$ and regard ${\rm DUAL}^2$ as a full subcategory of  ${\rm ALG}^5$.

Note that if $(\mathcal{B},\beta)\in {\rm DUAL}^2$, and $\mathcal{B'}$ is related to $\mathcal{B}$ via a simple homotopy equivalence (recall definition from \S\ref{intro}) then $(\mathcal{B},\beta)$ is related to $(\mathcal{B}',\beta)$ by a pair of simple homotopy equivalences.

Recall Definition  \ref{stabdef}.  Then $(\mathcal{B},\beta)$ and $(\mathcal{B}^n,\beta)$ are related by a simple homotopy equivalence, (where in the latter $\beta$ is understood to extend to the standard inner product on the free module which we regard as a direct sum of copies of $\ZG$). 

In some sense $\mathcal{B}$ is not important.  That is we may fix an algebraic 2--complex $\mathcal{B}$,  so that up to isomorphism any element of ${\rm DUAL}^2$ has the form $(\mathcal{B}^n,\beta)$ for some integer $n$ and bilinear form $\beta$.  To that end fix an algebraic 2--complex:
\[
\mathcal{B} :=\quad  J^* \stackrel{}{\dashrightarrow} F_2 \stackrel{d_2}{\to} F_1 \stackrel{d_1}{\to} F_0 \dashrightarrow \Z,\]
with $F_0,F_1,F_2$ fg free and exactness at $F_1$.  Let $J_n$ denote $J \oplus \ZG^n$.

\begin{thm}\label{4.1}Given $(\mathcal{C},\gamma)\in {\rm DUAL}^2$ there exists an integer $n$ and bilinear form $\beta$ on $J_n$ such that $(\mathcal{C},\gamma)$ is chain homotopy equivalent to $(\mathcal{B}^n,\beta)$.
\end{thm}

\begin{proof}  By \cite[Lemma 2.1]{Mann2} there exists an integer $n$ such that we may apply a pair of simple homotopy equivalences to $\mathcal{B}^n$ to get a complex{\rm :}
\[L_2 \stackrel{D_2}{\longrightarrow} L_1  \stackrel{D_1}{\longrightarrow} L_0,\]
and a pair of simple homotopy equivalences to some stabilization of $\mathcal{C}$ to get{\rm:}  
\[S_2 \stackrel{\partial_2 }{\longrightarrow}  S_1 \stackrel{\partial_1}{\longrightarrow}S_0,\]
such that there is  a chain isomorphism:
\begin{eqnarray*}
S_2 \stackrel{\partial_2 }{\longrightarrow}  S_1 \stackrel{\partial_1}{\longrightarrow}S_0\,\,\\ 
\downarrow \theta_2    \quad      \downarrow \theta_1 \quad   \downarrow \theta_0\\
L_2\stackrel{D_2}{\longrightarrow}  L_1 \stackrel{D_1}{\longrightarrow}L_0\,\,
\end{eqnarray*}
As noted above, $(\mathcal{C},\gamma)$ is chain homotopy equivalent to:
\[
S_0^* \stackrel{\partial_1^*}{\longrightarrow}  S_1^*  \stackrel{\partial_2^*}{\longrightarrow} S_2^* 
\stackrel{\partial_3}{\longrightarrow} S_2   \stackrel{\partial_2}{\longrightarrow} S_1  \stackrel{\partial_1}
{\longrightarrow} S_0, 
\]
for some map $\partial_3$.  Let $D_3 = \theta_2 \circ \partial_3 \circ \theta_2^*$.  We have a chain isomorphism:
\begin{eqnarray}
S_0^* \stackrel{\partial_1^*}{\longrightarrow}  S_1^*  \stackrel{\partial_2^*}{\longrightarrow}\,\,\, S_2^* \,\,\,
\stackrel{\partial_3}{\longrightarrow} \,\, S_2 \,\,\,  \stackrel{\partial_2}{\longrightarrow}\,\,\, S_1  \stackrel{\partial_1}
{\longrightarrow} S_0 \,\,\,\,\\ 
\quad \downarrow \theta_0^{*-1}\,\,\, \downarrow \theta_1^{*-1}
\,\,\, \downarrow \theta_2^{*-1} \qquad   \downarrow \theta_2    \quad \quad      \downarrow \theta_1 \quad   \downarrow \theta_0\nonumber\\
L_0^* \stackrel{D_1^*}{\longrightarrow}  \,\,L_1^*  \stackrel{D_2^*}{\longrightarrow}\,\, L_2^* \,\,\,
\stackrel{D_3}{\longrightarrow} \,\, L_2 \,\,\,   \stackrel{D_2}{\longrightarrow} \,\,\,L_1  \stackrel{D_1}
{\longrightarrow} L_0 
\,\,\,\,\label{lowerfive}
\end{eqnarray}
Finally note that (\ref{lowerfive}) is obtained from $(\mathcal{B}^n,\, \beta)$ for some $\beta$, by two pairs of simple homotopy equivalences.  Hence $(\mathcal{B}^n,\, \beta)$ is chain homotopy equivalent to $(\mathcal{C},\gamma)$.
 \end{proof}

So with $\mathcal{B}$ fixed, for any $M \in  {\rm TOP}^5$ we have an integer $n$ and bilinear form $\beta$ on $J_n$, such that $C(M)$ is isomorphic to $(\mathcal{B}^n,\beta)$ in ${\rm ALG}^5$.  In particular the form $\beta$ encapsulates all the twisted (co)homology of $M$.  A benefit of this is that one may attempt to classify objects in ${\rm ALG}^5$ and ${\rm TOP}^5$ without having to classify algebraic 2--complexes over $\ZG$ (an abstruse problem in its own right \cite{John1}).

\section{Anti-self-duality}\label{Poincaresec}

Pick $\eta\in H_5(C(M))$ a generator.  Poincar\'e Duality may be stated:  {\it There is a  chain homotopy equivalence $\phi\colon C(M)^* \to C(M)$, such that for $\alpha \in H^p(C(M))$ we have}{\rm:} \begin{eqnarray}\phi_*(\alpha) =-1^{p(p+1)/2} \eta \,\,\,\widehat{}\,\,\, \alpha. \label{chaindef}\end{eqnarray}

Given $\alpha \in H^p(C(M))$, $\gamma \in H^{5-p}(C(M))$ the graded symmetry of the cup product implies that $\gamma(\eta \,\,\,\widehat{}\,\,\, \alpha)=\alpha(\eta \,\,\,\widehat{}\,\,\, \gamma)$.  Substituting (\ref{chaindef}) into this we get: \begin{eqnarray} \label{chaindef1} \gamma(\phi_*(\alpha)) =- \alpha(\phi_*(\gamma)).\end{eqnarray}

Suppose, that we have a chain homotopy equivalence $f\colon C(M) \to \mathcal{A}$, for some $\mathcal{A}\in{\rm ALG}^5$.  We have a chain homotopy equivalence $\phi'=f \circ \phi \circ f^*\colon \mathcal{A}^* \to \mathcal{A}$.  

\begin{defn} If a homotopy equivalence $\mathcal{A}^* \to \mathcal{A}$ is chain homotopy equivalent to one constructed in this
way for some $f$, we say it is a \emph{duality equivalence}.
\end{defn}

Given $\alpha \in H^p(\mathcal{A})$, $\gamma \in H^{5-p}(\mathcal{A})$ we have $\gamma(\phi'_*(\alpha)) = f^*(\gamma)(\phi_*f^*(\alpha))$ which is antisymmetric in $\alpha,\gamma$ (by (\ref{chaindef1})), so we may conclude
$\gamma(\phi'_*(\alpha)) = -\alpha(\phi'_*(\gamma))$.

By Theorem \ref{algtodual} we may choose $\mathcal{A}$ to be of the form $(\mathcal{B},\beta)$ for some: 
\[\mathcal{B}= J^* \stackrel{\iota}\dashrightarrow F_2 \stackrel{d_2} {\to}F_1 \stackrel{d_1}{\to} F_0 \dashrightarrow Z,\] and bilinear form $\beta\colon  J \times J \to \Z$.  Let $d_3=\iota\beta\iota^*$.  

The duality equivalence $\phi'\colon(\mathcal{B},\beta)^*\to (\mathcal{B},\beta)$ may be written: 
\begin{eqnarray}
\Z \dashrightarrow F_0^* \stackrel{d_1^*} {\to}F_1^* \stackrel{d_2^*}{\to} F_2^* \stackrel{d_3^*}{\to}
 F_2 \stackrel{d_2} {\to}F_1 \stackrel{d_1}{\to} F_0 \dashrightarrow \Z\,\,\,\notag\\
y \downarrow \,\,\,\,\,\,\, \phi'_5\downarrow
\,\,\,\,\phi'_4\downarrow \,\,\,\,\phi'_3\downarrow\,\,\, \phi'_2\downarrow \,\,\,
\phi'_1\downarrow\,\,\, \phi'_0 \downarrow\quad\,  x\downarrow \,\,\, \label{dequiv1}\\
\Z \dashrightarrow F_0^* \stackrel{d_1^*} {\to}F_1^* \stackrel{d_2^*}{\to} F_2^* \stackrel{d_3}{\to}
 F_2 \stackrel{d_2} {\to}F_1 \stackrel{d_1}{\to} F_0 \dashrightarrow \Z\,\,\,\notag
\end{eqnarray}
The natural pairing $F_0^* \times F_0 \to \Z$ induces a pairing on the copies of $\Z$ on the left of this diagram, with those on the right.  We set the generators $1\in \Z$ on the right and left so that this pairing is given by multiplication.  Let $\alpha\in H^5((\mathcal{B},\beta)),\, \gamma \in H^0((\mathcal{B},\beta))$ correspond to these generators.

We have $x=\gamma(\phi'_*(\alpha)) = -\alpha(\phi'_*(\gamma))=-y$.  Without loss of generality we may assume (by for example replacing $d_3$ with $-d_3$), that $x=1, y=-1$.

\begin{thm} \label{Almost1s} For some maps $\theta_1$, $\theta_2$ we have a duality equivalence{\rm:}
\begin{eqnarray} 
F_0^* \stackrel{d_1^*}{\longrightarrow} F_1^* \stackrel{d_2^*}{\longrightarrow} F_2^* \stackrel{d_3^*}{\longrightarrow} 
F_2 \stackrel{d_2}{\longrightarrow} F_1 \stackrel{d_1}{\longrightarrow} F_0\,\,\,\,\, \nonumber\\
\downarrow -1  \quad\downarrow -1 \,\,\,\,\,\,\downarrow \theta_2 \,\,\,\quad \downarrow \theta_1 \,\,\quad\downarrow 1
\qquad\downarrow 1 \label{duality4outof6}\\
F_0^* \stackrel{d_1^*}{\longrightarrow} F_1^* \stackrel{d_2^*}{\longrightarrow} F_2^* \stackrel{d_3}{\longrightarrow} 
F_2 \stackrel{d_2}{\longrightarrow} F_1 \stackrel{d_1}{\longrightarrow} F_0\,\,\,\,\, \nonumber
\end{eqnarray}
\end{thm}

\begin{proof}
By the projectivity of the $F_i$ we may define maps $I_0,I_1,I_3,I_4$: 

\xymatrix{
J^*\ar[r]\ar[d]&F_2\ar[d]_{1-\phi_2'}\ar[r]^{d_2}&F_1\ar[dl]_{I_1}\ar[d]_{1}^{-\phi_1'}\ar[r]^{d_1}&F_0\ar[dl]^{I_0}\ar[d]^{1-\phi_0'}\ar[r]&\Z\ar[d]^0&
J^*\ar[d]\ar[r]&F_2\ar[d]_{-1-\phi_3'^*}\ar[r]^{d_2}&F_1\ar[dl]_{I_3^*}\ar[d]_{-1}^{-\phi_4'^*}\ar[r]^{d_1}&F_0\ar[dl]^{I_4^*}\ar[d]^{-1-\phi_5'^*}\ar[r]&\Z\ar[d]^0
\\
J^*\ar[r]&F_2\ar[r]^{d_2}&F_1\ar[r]^{d_1}&F_0\ar[r]&\Z&
J^*\ar[r]&F_2\ar[r]^{d_2}&F_1\ar[r]^{d_1}&F_0\ar[r]&\Z
}

\noindent such that:
\begin{eqnarray*}
d_1 I_0 =1- \phi_0'&,& \qquad I_0 d_1+d_2 I_1 =1- \phi_1',\\
d_1 I_4^* =-1- \phi_5'^*&,& \qquad I_4^* d_1+d_2 I_3^* =-1- \phi_4'^*.
\end{eqnarray*}
Set    \begin{eqnarray}\theta_1= \phi_2'+I_1d_2,\qquad \theta_2= \phi_3'+ d_2^*I_3,\qquad I_2=0.\label{I2is0}\end{eqnarray}
Then the $I_i$, $i=0,1,2,3,4$ form the required chain homotopy (\ref {dequiv1}) to (\ref{duality4outof6}).
 \end{proof}

Theorem \ref{algtodual} tells us that we may always  pick a representative $\mathcal{A}=(\mathcal{B},\beta)$ of the homotopy type of $C(M)$ such that the modules of $\mathcal{A}$ and $\mathcal{A}^*$ may be identified, with four of the five differentials  identical.  Theorem  \ref{Almost1s} tells us that Poincar\'e duality induces a homotopy equivalence $\mathcal{A}^*\to\mathcal{A}$ which is plus/minus the identity on four of the six maps.  The natural completion  of these results would be that we could also take $\theta_1=1,\,\theta_2=-1$ (which would imply $d_3^*=-d_3$).

\begin{defn}
Say that $M \in {\rm TOP}^5$ is \emph{anti-self-dual} if we can pick $\mathcal{A}$ as above with  $\theta_1=1,\,\theta_2=-1$.
\end{defn}

Indeed in (\ref{I2is0}) we arbitrarily set $I_2=0$, so it may seem that a more careful choice of $I_2$ could lead to $\theta_1=1,\,\theta_2=-1$.  There are two further reasons to conjecture that all $M\in {\rm TOP}^5$ are anti-self dual.   These arise in the two themes which we seek to marry in this paper: the derived category of the fundamental group of $M$ and the homology of the  simply connected manifold $\tilde{M}$. 

Let $\lambda_1\colon J \to J$ be the dual of the map induced by $\theta_1$ and let $\lambda_2\colon J \to J$ be the map induced by $\theta_2$.  Commutativity in the central square of (\ref{duality4outof6}) implies that for all $u, v \in J$: $\beta(\lambda_1 u,  v)= \beta(\lambda_2v, u).$

Thus if we had $\lambda_1=1_J,\,\lambda_2=-1_J$ then $\beta$ would be antisymmetric.  What we do know from Theorem \ref{Almost1s} is that this is true at level of the derived category of $G$ (in the sense of \cite[Chapter 4]{John1});  $\lambda_1,\lambda_2$ augment to $1,-1 \in \Z/\vert G \vert$ respectively.   

On the other hand Tor$(H^3(\tilde{M}))$ may be naturally identified with a subquotient of $J\otimes \Q$.  With respect to this identification, the linking form on  Tor$(H^3(\tilde{M}))$ (taking values in $\Q/\Z$) is induced by $\beta(\lambda_1\_,\,\_)$ \cite{Bard,Wall1}.  From {\rm \cite[Lemma D(ii)]{Bard}}  we know that this is antisymmetric.

Despite all this, it transpires that not all manifolds are anti-self-dual.  An obstruction arises from a combination of a condition on $G$ and a condition on $H_*(\tilde{M})$.

\begin{thm}\label{obstruct}
Let $M \in {\rm TOP}^5$. If $G$ has even order and $H_3(\tilde{M})$ has even rank over $\Z$, then $M$ is not anti-self-dual.
\end{thm}

\begin{proof}
Suppose $M$ is anti-self-dual.  Then we have that $C(M)$ is chain homotopy equivalent to $(\mathcal{B},\beta)$, with  $\beta$ an antisymmetric form on a module $J$ with $J^*\in \Omega_3(\Z)$.  Thus $\rk(J) \equiv -1 \mod |G|$, and in particular is odd.
Now $H_3(\tilde{M})$ may be identified naturally with the kernel of $\beta$.  

Thus we may split (over $\Z$) $J=H_3(\tilde{M})\oplus J'$ and extend $\beta$ to the tensor product $V=J' \otimes \Q$.  Then $V$ is an odd dimensional vector space over $\Q$, equipped with a non-degenerate anti-symmetric bilinear form.  This is impossible.
 \end{proof}

As an example let $M$ be the Lens space $L(n;1,1)$.  This has fundamental group  $C_n=\langle t \vert\,\, t^n=1\rangle$.  Let $\Sigma$ denote the sum of the $t^i$.  We have $\mathcal{A}$ in dual form, with a chain homotopy equivalence $C(M)\to \mathcal{A}$ yielding the duality morphism $\phi$:

\xymatrix{
\mathcal{A^*}\!:=\ar[d]_\phi&\Z[t]\ar[d]_{-1} \ar[r]^{\times(1-t^{-1})} &\Z[t]\ar[d]_{-1} \ar[r]^{\times \Sigma}&\Z[t]\ar[d]_{-t} \ar[r]^{\times(1-t)}&
\Z[t]\ar[d]_{1}\ar[r]^{\times \Sigma}& \Z[t]\ar[d]_{1} \ar[r]^{\times(1-t)}&\Z[t]\ar[d]_{1}
\\
\mathcal{A}\,:=&\Z[t] \ar[r]^{\times(1-t^{-1})} &\Z[t] \ar[r]^{\times \Sigma}&\Z[t] \ar[r]^{\times(1-t^{-1})}&
\Z[t]\ar[r]^{\times \Sigma}& \Z[t] \ar[r]^{\times(1-t)}&\Z[t]
}

Note that dualizing multiplication by an element of $\Z[t]$ corresponds to involuting the element (e.g. substituting $t^{-1}$ for $t$). This results in the change in sign between the components of $\phi$ on the left and right sides, implied by Theorem  \ref{Almost1s}.  

Note that the fourth component (from the right) of $\phi$ is $-t$, not $-1$.  If $n$ is even then this is unavoidable; $\tilde{M}=S^5$ so we know $H_3(\tilde{M})=0$ and Theorem  \ref{obstruct} implies that $M$ is not anti-self-dual.  Conversely:

\begin{prop}
If $n=4k+1$ then $M$ is anti-self-dual.
\end{prop}

\begin{proof}
\[\,\,\,\,{\rm Let} \qquad 
\alpha=t^{k+1}+t^k-t^{-k}-t^{-(k+1)},\qquad \beta =\sum_{r=-k+1}^k t^r -\sum_{r=k+2}^{3k}t^r.
\]
Then $\alpha$ generates the same ideal as $t-1$, as $\alpha(t^{1+k}+t^{1-k})=t^2-1$, which when multiplied by $1+t^2+t^4+\cdots+t^{4k}$ returns $t-1$.

Now $\beta(1-t^{-1})=\alpha$ and $\Sigma\beta=\Sigma$.  Thus $\gamma\beta(1-t^{-1})=(1-t^{-1})$ for some $\gamma$, and $\gamma\beta=1+s\Sigma$ for some integer $s$.  Then $(\gamma-s\Sigma)\beta=1$ and $\beta$ is a unit. Thus we have a chain homotopy equivalence $f\colon \mathcal{A}\to \mathcal{A'}$:

\xymatrix{
\mathcal{A}\,:=\ar[d]_f&\Z[t]\ar[d]_{1} \ar[r]^{\times(1-t^{-1})} &\Z[t]\ar[d]_{1} \ar[r]^{\times \Sigma}&\Z[t]\ar[d]_{1} \ar[r]^{\times(1-t^{-1})}&
\Z[t]\ar[d]_{\beta}\ar[r]^{\times \Sigma}& \Z[t]\ar[d]_{1} \ar[r]^{\times(1-t)}&\Z[t]\ar[d]_{1}
\\
\mathcal{A'}\,:=&\Z[t] \ar[r]^{\times(1-t^{-1})} &\Z[t] \ar[r]^{\times \Sigma}&\Z[t] \ar[r]^{\times\alpha}&
\Z[t]\ar[r]^{\times \Sigma}& \Z[t] \ar[r]^{\times(1-t)}&\Z[t]
}

The chain homotopy equivalence $f\phi f^*$ is chain homotopic to the chain map whose first three components are 1 and whose last three components are -1.  The required homotopy has only one non-zero component; $I_2=x$, where $\alpha x=\beta-1$.
 \end{proof}

\end{document}